\documentclass[11pt, a4paper, oneside]{article}
\usepackage{amsmath, amssymb, amsthm, amsfonts, amsxtra, latexsym, amscd,
pb-diagram, enumerate, graphics}
\usepackage [all]{xy}

\newcommand{\ri}{\rightarrow }

\newcommand{\Fm}{\widetilde{F}}

\DeclareMathOperator{\Coker}{Coker} 
\DeclareMathOperator{\Hom}{Hom} 
\DeclareMathOperator{\Aut}{Aut} \DeclareMathOperator{\Ext}{Ext}
\DeclareMathOperator{\Ker}{Ker} 
\DeclareMathOperator{\Dis}{Dis} \DeclareMathOperator{\Ob}{Ob}

\newtheorem{theorem}{\bf Theorem}[section]
\newtheorem{lemma}[theorem]{\bf Lemma} % Lemma
\newtheorem{proposition}[theorem]{\bf Proposition} % Proposition
\newtheorem{corollary}[theorem]{\bf Corollary} % Corollary
\newtheorem{example}[theorem]{\bf Example} % Corollary
\newtheorem{definition}[theorem]{\bf Definition} % Corollary
\newtheorem{remark}[theorem]{\bf Remark} % Corollary

\begin{document}

% Title of document, usually lower case except for first word
% and proper nouns.  Avoid unnecessary symbols.
\centerline{\Large\bf Braided equivariant crossed modules and}
\centerline{\Large\bf cohomology of $\Gamma $-modules}

% If the title is too long for the running head, use
% the following command to specify a short title:
%\shorttitle{Braided equivariant crossed modules and cohomology of $\Gamma $-modules}

% First Author
 \vspace{0.5cm}
\centerline{\textsc{Nguyen. T. Quang$^{1,*}$, Che. T. K.
Phung$^{2}$, Pham. T. Cuc$^{3}$}}
 \vspace{0.5cm}
  \noindent\textit{ $^1$Department of Mathematics, Hanoi National University of
Education,  Vietnam}

\noindent \textit{$^2$Department of Mathematics and Applications,
Saigon University,  Vietnam}

\noindent \textit{$^3$Natural Science Department, Hongduc
University,  Vietnam}
\renewcommand{\thefootnote}{}

\footnote{$^*$ Corresponding author.  \textit{Email adrdresses:}
cn.nguyenquang@gmail.com (Nguyen. T. Quang), kimphungk25@yahoo.com
(Che. T. K. Phung), cucphamhd@gmail.com (Pham. T. Cuc)}

 \vspace{0.5cm}
% Additional authors done in the same way.

% If the author names are too long for the running head, use
% the following command to specify a shorter version:
%\shortauthors{DOE, SMITH \andname\ WILLIAMS}

% AMS 2000 Mathematics Subject Classification.  List one or several,
% separated by commas, ending in a period.

% Abstract comes before maketitle
\begin{abstract}
% Abstract text, usually no more than 200 words.
% Avoid bibliographic references (\cite) and complicated mathematics.
% Please do not use custom macros here, as this abstract has to
% be able to stand alone.  You may use standard tex/latex/AMS macros.
If $\Gamma $ is a group, then braided $\Gamma $-crossed modules are
classified by  braided strict $\Gamma $-graded categorial groups.
The Schreier theory obtained for $\Gamma $-module extensions of the
type of an abelian $\Gamma $-crossed module is a generalization of
the  theory of $\Gamma $-module extensions.
\end{abstract}

\noindent{\small{\bf 2010 Mathematics Subject Classification:}
18D10, 20J05, 20J06, 20E22}

\noindent {\small{\bf Keywords:} braided $\Gamma $-crossed module,
braided strict graded  categorical group,  $\Gamma $-module
extension, symmetric cohomology}
% Text of Document.  Use constructs such as \section, \subsection,
% \begin{theorem} ... \end{theorem}, \begin{proof} ... \end{proof}, etc.

\section{Introduction}
Crossed modules have been used widely, and in various contexts, since their definition by Whitehead \cite{White49}
in his investigation of the algebraic structure of second relative homotopy groups.
Brown and Spencer \cite{Br76} showed that crossed modules are determined  by \emph{$\mathcal G$-groupoids} (or {\it strict categorial groups}), and hence  crossed modules can be studied by the theory of category.
Thereafter, Joyal and Street \cite{J-S} extended the result in \cite{Br76} for \emph{ braided} crossed modules and \emph{braided} strict categorial groups. A {\it braided strict categorial group} is a braided categorical group in which the unit, associativity constraints are strict
 and every object is invertible ($x\otimes y=1=y\otimes x$).

 A brief summary of researches
 related to crossed modules was given in \cite{CCG} by  Carrasco et al. Results on the category of \emph{abelian} crossed modules appeared in this work. Previously, the notion of abelian crossed module was characterized by that of the \emph{center} of a crossed module in  the paper of  Norrie  \cite{Norrie}.

In \cite{Fro},  Fr\"{o}hlich and Wall  introduced the notion of
graded categorical group.
 Thereafter,
 Cegarra and Khmaladze   constructed the abelian (symmetric) cohomology of $\Gamma $-modules
  which was applied on the classification for braided (symmetric) $\Gamma $-graded categorical
  groups in \cite{CK} (\cite{CK2007}).

The purpose of this paper is to study kinds of crossed modules which
are defined by braided strict $\Gamma $-graded  categorical groups.
This result is an extension of the result of  Joyal and Street
mentioned above. After this introductory Section 1, Section 2 is
devoted to recalling some necessary fundamental notions and results
of braided (symmetric) graded categorical groups and factor sets of
braided graded categorical groups. In Section 3 we show that the
category ${\bf BrGr^\ast}$ of braided strict categorical groups and
regular symmetric monoidal functors is equivalent to the category
\textbf{BrCross} of braided crossed modules (Theorem \ref{pl}). Each
morphism in the   category ${\bf BrCross}$ consists of a
homomorphism $(f_1,f_0):\mathcal M\ri \mathcal M'$ of braided
crossed modules  and an element of the group of abelian 2-cocycles
$Z^2_{ab}(\pi_0\mathcal M, \pi_1\mathcal M')$ in the sense of
\cite{E-M}. This result is a continuation of  the result in
\cite{J-S} (Remark 3.1). It is obtained as a consequence of
Classification Theorem \ref{dlpl}.

In Section 4 we extend the result in Section 3 to graded structures
by introducing
 the notions of {\it braided $\Gamma$-crossed module} and
  \emph{braided strict $\Gamma $-graded  categorical group} to classify braided
  $\Gamma$-crossed modules (see \cite{N}).
   Theorem  \ref{dlpl} states that
 the category ${\bf_{\Gamma}BrGr^\ast}$ of braided strict $\Gamma $-graded  categorical groups and
regular $\Gamma $-graded symmetric  monoidal functors is equivalent
to the category $\bf {_{\Gamma}BrCross}$ of braided $\Gamma
$-crossed modules.  Each morphism in the   category $\bf
{_{\Gamma}BrCross}$ consists of a homomorphism $(f_1,f_0):\mathcal
M\ri \mathcal M'$ of braided $\Gamma$-crossed modules  and an
element of the group of symmetric 2-cocycles
$Z^2_{\Gamma,s}(\pi_0\mathcal M, \pi_1\mathcal M')$ in the sense of
\cite{CK}.

The problem of group extensions of the type of a crossed module has
been mentioned in \cite{Tay, Ded, Br94}. In Section 5 we show a
treatment of the similar problem for  $\Gamma$-module extensions of
the type of an abelian $\Gamma$-crossed module. The Schreier theory
for such extensions (Theorem \ref{schr}) is presented by means of
graded symmetric  monoidal functors, and therefore we obtain the
classification theorem of $\Gamma $-module extensions  of the type
of an abelian $\Gamma$-crossed module (Theorem \ref{dlc}).

The case of (non-braided)  $\Gamma$-cossed modules  is studied by
Quang and Cuc in \cite{Q-C}. The results   generalizes both the
theory of group extensions of the type of a crossed module and the
one of equivariant group extensions.

\section{Preliminaries}

\subsection{Braided (symmetric) graded categorical groups}
Let $\Gamma$ be a fixed group, which we regard as a category with
exactly one object, say $*$, where the morphisms are the members of
$\Gamma $ and the composition law is the group operation. A {\it
grading} on a category $\mathbb G$ is then  a functor, say
$gr:\mathbb G\ri \Gamma$.
 For any morphism $u$ in $\mathbb G$ with $gr(u)=\sigma$, we refer to $\sigma$ as the {\it grade} of $u$.
The grading $gr$ is said to be \emph{stable} if for any
$X\in$ Ob$\mathbb G$ and any $\sigma\in \Gamma$ there exists an  isomorphism
$u$ in $\mathbb G$ with domain  $X$ such that $gr(u)=\sigma$.

 A  \emph{braided $\Gamma $-graded monoidal category} \cite{CK}  $\mathbb G=(\mathbb G, gr, \otimes ,$ $I,  \textbf{a, r, l, c})$ consists of a category $\mathbb G$, a stable grading  $gr: \mathbb G \rightarrow \Gamma ,$ graded functors $\otimes : \mathbb G\times _\Gamma \mathbb G\rightarrow \mathbb G$ and $I: \Gamma \rightarrow \mathbb G$, and graded natural equivalences defined by isomorphisms of grade 1  ${\bf a}_{X,Y,Z}: (X\otimes Y)\otimes Z
\stackrel{\sim}{\rightarrow}
 X\otimes (Y\otimes Z ),{\bf l}_X: I\otimes X \stackrel{\sim}{\rightarrow} X,
 {\bf r}_X: X\otimes I \stackrel{\sim}{\rightarrow} X$ and ${\bf c}_{X,Y}: X\otimes
  Y\stackrel{\sim}{\rightarrow} Y\otimes X$ satisfying the following coherence conditions:
\begin{equation*}
{\bf a}_{X,Y,Z\otimes T}{\bf a}_{X\otimes Y, Z,T}= (id_X \otimes
{\bf a}_{Y,Z,T})
{\bf a}_{X,Y\otimes Z,T}({\bf a}_{X,Y,Z}\otimes id_T), \label{2.1}
 \end{equation*}
 \begin{equation*}
(id_X \otimes {\bf l}_Y){\bf a}_{X,I,Y}= {\bf r}_X \otimes id_Y, \label{2.2}
 \end{equation*}
 \begin{equation}
(id_Y \otimes {\bf c}_{X,Z}) {\bf a}_{Y,X,Z}({\bf c}_{X,Y}\otimes
id_Z)={\bf a}_{Y,Z,X}
{\bf c}_{X,Y\otimes Z} {\bf a}_{X,Y,Z}, \label{2.3}
 \end{equation}
 \begin{equation}
({\bf c}_{X,Z}\otimes id_Y){\bf a}^{-1}_{X,Z,Y}(id_X \otimes {\bf
c}_{Y,Z}) = {\bf a}_{Z,X,Y}^{-1}{\bf c}_{X\otimes Y,Z}{\bf
a}^{-1}_{X,Y,Z}. \label{2.4}
 \end{equation}
 A {\it braided $\Gamma $-graded categorical group} \cite{CK} is a braided $\Gamma $-graded monoidal groupoid  such that, for
 any object $X$, there is an object
  $X'$ with an arrow $X\otimes X'\ri 1$ of grade 1. If the braiding ${\bf c}$ is a symmetric constraint, that is, it satisfies the condition ${\bf c}_{Y,X}\circ
 {\bf c}_{X,Y}=id_{X\otimes Y}$ (in this case  the relation (\ref{2.4}) coincides with the relation (\ref{2.3})), then $\mathbb G$ is called a {\it symmetric} $\Gamma $-graded categorical group
 or a {\it graded Picard category}  \cite{CK2007}. Then the subcategory
 $\Ker \mathbb G$ (whose objects are the objects of $\mathbb G$ and morphisms are the morphisms of grade 1 in  $\mathbb G$)
  is a  braided categorical group (a Picard category, respectively).

Let  $(\mathbb G,gr)$ and $(\mathbb G',gr')$ be two (braided
symmetric) $\Gamma $-graded categorical groups. A {\it graded
symmetric monoidal functor}
 from $(\mathbb G,gr)$ to $(\mathbb G',gr')$ is a triple
$(F,\widetilde{F},F_\ast)$, where
 $F : (\mathbb G,gr) \rightarrow (\mathbb G',gr')$ is a
 $\Gamma$-graded functor,
 $\widetilde{F}_{X,Y} :FX \otimes FY \rightarrow F(X\otimes Y)$
are natural isomorphisms of grade 1 and
 $F_\ast : I' \rightarrow FI$  is an isomorphism of grade 1, such that
  the following coherence conditions hold:
\begin{equation*}
\widetilde{F}_{X,Y\otimes Z}(id_{FX}\otimes \widetilde{F}_{Y,Z})
{\bf a}_{FX,FY,FZ}= F({\bf a}_{X,Y,Z})\widetilde{F}_{X\otimes
Y,Z}(\widetilde{F}_{X,Y}\otimes id_{FZ}), \label{6.5}
\end{equation*}
\begin{equation*}
F({\bf r}_X)\widetilde{F}_{X,I}(id_{FX}\otimes F_\ast)= {\bf
r}_{FX},\
 F({\bf l}_X)\widetilde{F}_{I,X}(F_\ast\otimes id_{FX}) = {\bf l}_{FX}, \label{6.6}
 \end{equation*}
 \begin{equation*}
\widetilde{F} _{Y,X}{\bf c}_{FX,FY}=F({\bf
c}_{X,Y})\widetilde{F}_{X,Y}. \label{6.7}
\end{equation*}

Let $(F,\widetilde{F},F_\ast), (F',\widetilde{F}',F'_\ast)$
be two graded symmetric monoidal functors. A {\it
graded symmetric monoidal natural equivalence}
$\theta:F\stackrel{\sim}{\rightarrow} F'$ is a graded natural equivalence such that, for all objects $X,Y$ of $\mathbb G$, the following coherence conditions hold
\begin{align}\label{3.4}
\widetilde{F}'_{X,Y}(\theta_X\otimes\theta_Y)=\theta_{X\otimes
Y}\widetilde{F}_{X,Y},\;\theta_I F_\ast=F'_\ast,
\end{align}
that is, a monoidal natural equivalence.
%----------------------------------------------

\subsection{Braided (symmetric) graded categorical groups of type $(M,N)$ and the theory of obstructions}

Let $\mathbb G$ be a  braided $\Gamma $-graded categorical group. We
write $M = \pi_0(\Ker \mathbb G)$ $= \pi_0\mathbb G$ for the abelian
group of 1-isomorphism classes of the objects in  $\mathbb{G}$ and
$N=\pi_1(\Ker \mathbb G)=\pi_1\mathbb G$ for the abelian group of
1-automorphisms of the unit object of $\mathbb G$. Then $\mathbb G$
induces  $\Gamma $-module structures on $M,N$ and a normalized
3-cycocle  $h\in Z^3_{\Gamma,ab}(M,N)$ in the sense of  \cite{CK}.
From these data, the authors of  \cite{CK} constructed a braided
$\Gamma $-graded categorical group, denoted by $\mathbb G(h)$ (or
$\int_{\Gamma}(M,N,h)$), which is equivalent to $\mathbb G$. Below,
we briefly recall  this construction.

The objects of $\mathbb G(h)$ are the elements $s\in M$ and their arrows are pairs $(a,\sigma):r\ri s$ consisting of an element $a\in
N$ and an element $\sigma\in\Gamma$ with $\sigma r=s$.

The composition of two morphisms $(r\xrightarrow{(a,\sigma) }s
\xrightarrow{(b,\tau )}t)$ is defined by
\begin{align*}\label{ht}(b,\tau )\circ(a,\sigma)=(b+\tau a+h(r,\tau,\sigma),\tau\sigma).  \end{align*}

The graded tensor product
is defined by
\begin{align*} (r\stackrel{(a,\sigma) }{\rightarrow}s )\otimes (r'\stackrel{(b,\sigma) }{\rightarrow} s')=
(rr'\xrightarrow{(a+sb+h(r,r',\sigma),\sigma) }ss' ).
\end{align*}

The unit constraints are strict,
$\mathbf{l}_s=(0,1)=\mathbf{r}_s:s\ri s$.
 The associativity and braiding constraints  are, respectively, given by
\[\mathbf{a}_{r,s,t}=(h(r,s,t),1):(rs)t\ri r(st),\]
\[\mathbf{c}_{r,s}=(h(r,s),1): rs \rightarrow sr.\]
The stable $ \Gamma$-grading is defined by
 $gr(a,\sigma)=\sigma$. The unit graded functor $I: \Gamma \rightarrow
\mathbb G(h)$ is defined by
\begin{align*} I(*\xrightarrow{\sigma } *)=(1\xrightarrow{(0,\sigma) } 1). \end{align*}
We call  $\mathbb G(h)$ a
 {\it reduced} braided $\Gamma $-graded categorical group of  $\mathbb G$.
  In the case when $\mathbb G$ is a $\Gamma $-graded Picard category, then $h\in Z^3_{\Gamma ,s}$ in the sense of
  \cite{CK2007} and $\mathbb G(h)$ is a  $\Gamma $-graded Picard category.

Let $\mathbb G$, $ \mathbb G'$ be $\Gamma $-graded Picard
categories, and let $\mathbb  G(h)= \int_{\Gamma}(M,N,h)$, $\mathbb
G'(h')= \int_{\Gamma}(M',N',h')$ be their reduced $\Gamma $-graded
Picard categories, respectively. A graded functor  $F: \mathbb
G(h)\ri \mathbb   G'(h')$ is said to be \emph{of type $(\varphi,f)$}
if
\begin{equation*}
F(s)=\varphi(s),\ \ F(a,\sigma)=(f(a),
\sigma),\;\;s\in M,\;a\in N,\;\sigma\in\Gamma,
\end{equation*}
 where $\varphi:M\ri M';\; f: N\rightarrow N'$ are homomorphisms of $\Gamma $-modules. Then the function
  \begin{equation*}
 k=\varphi^{\ast}h'-f_{\ast}h
 \end{equation*}
 is called an  {\it obstruction} of  the functor $F.$

% Each graded symmetric monoidal functor
%  $F:\mathbb G\ri\mathbb G'$ induces a pair of $\Gamma $-module homomorphisms
%$(\varphi,f):\mathbb
%G(h)\ri \mathbb G'(h')$ given by
%\begin{gather*}
% \varphi=F_0: \pi_{0}\mathbb G\rightarrow \pi_{0}\mathbb G^{'}, \ \ [X]\mapsto [FX],\\
%f=F_1:\pi_{1}\mathbb G\rightarrow \pi_{1}\mathbb G^{'}, \ \ u\mapsto
%\hat{\gamma}^{-1}_{FI}(Fu).
%\end{gather*}

Based on the results on monoidal functors of  type
$(\varphi,f)$ presented in
\cite{QCT}, we obtain the following results with some appropriate modifications.

  \begin{proposition}\label{dl2.1a}
Let $\mathbb G, \mathbb G'$ be  braided $\Gamma $-graded categorical groups, and let $\mathbb  G(h)$, $\mathbb   G'(h')$ be their reduced braided $\Gamma $-graded categorical groups, respectively.\\
\indent \emph{i)}  Any graded symmetric monoidal functor $(F, \Fm): \mathbb G\rightarrow \mathbb G'$
induces a  graded symmetric  monoidal functor $\mathbb  G(h)\rightarrow \mathbb  G'(h')$ of type $(\varphi,f)$.\\
\indent \emph{ii)} Any   graded symmetric  monoidal functor $\mathbb
G(h)\rightarrow \mathbb  G'(h')$ is a
 graded functor of type $ ( \varphi, f).$
\end{proposition}
\begin{proposition} [\cite{CK2007}, Theorem 3.9] \label{dl2.2a} The graded functor  $(F,\widetilde{F}): \mathbb   G(h)\ri \mathbb   G'(h') $ of type
$(\varphi, f)$ is  realizable, that is,
 there are isomorphisms $\widetilde{F}_{x,y} $ so that $(F,\widetilde{F} )$ is  a graded symmetric monoidal functor, if and only if its obstruction $\overline{k} $ vanishes in  $H^3_{\Gamma,s}(M, N')$. Then, there is a bijection
 \begin{equation*}
     {\mathrm{Hom}}_{(\varphi, f)}[\mathbb   G(h), \mathbb   G'(h')]\leftrightarrow H^2_{\Gamma,s}(M,
 N'),
\end{equation*}
where $\Hom_{(\varphi, f)}[\mathbb G(h), \mathbb G'(h')]$
 denotes the set
 of homotopy classes of  graded symmetric   monoidal functors of type $(\varphi ,f)$ from
$\mathbb G(h)$ to $\mathbb G'(h')$.
\end{proposition}

Note that $H^2_{\Gamma,s}(M,N')=H^2_{\Gamma,ab}(M, N')$.
\subsection{ Factor sets in braided graded categorical groups}

According to the  definition of a factor set with coefficients in a  monoidal category
\cite{C2001}, we now establish the following terminology.
\begin{definition}  \emph{A} {\it symmetric factor set} \emph{$\mathcal F$ on $\Gamma$ with coefficients in a  braided categorical group $\mathbb G$
(or a pseudofunctor from $\Gamma$ to the category of braided
categorical groups in the sense of Grothendieck \cite{Gro}) consists
of a family of symmetric  monoidal  auto-equivalences
$F^{\sigma}:\mathbb G\rightarrow \mathbb G, \sigma\in\Gamma
 $, and isomorphisms between symmetric monoidal functors
$\theta^{\sigma,\tau} : F^{\sigma} F^{\tau} \rightarrow
F^{\sigma\tau} $, $ \sigma,\tau\in\Gamma$ satisfying the
conditions:}

 \emph{i)} $F^{1} = id_\mathbb G$,

 \emph{ii)} $ \theta^{1,\sigma} = id_{F^{\sigma}} = \theta^{\sigma,1}$, $\sigma \in\Gamma$,

 \emph{iii)} \emph{for all $ \sigma,\tau,\gamma\in\Gamma$, the following diagram
 commutes}
\begin{equation*}\label{ght3}\begin{CD}
F^{\sigma} F^{\tau} F^{\gamma} @> \theta^{\sigma,\tau} F^{\gamma} >> F^{\sigma\tau}F^{\gamma}\\
@V F^{\sigma}\theta^{\tau,\gamma} VV           @VV \theta^{\sigma\tau,\gamma}  V\\
F^{\sigma}F^{\tau\gamma} @> \theta^{\sigma,\tau\gamma}>>
F^{\sigma\tau\gamma}.
\end{CD}.\end{equation*}
\end{definition}
We write $\mathcal F=(\mathbb G,F^\sigma,\theta^{\sigma,\tau})$, or simply $(F,\theta)$.

%We obtain the following result which is a braided analogue of Theorem 2.1 in  \cite {C2001}.

The following lemma comes from an analogous result on graded
monoidal categories  \cite{C2001} or a part of Theorem 1.2 \cite{U}.
 We sketch the proof since we need some of its details.
 \begin{lemma}\label{bd01} Any braided $\Gamma $-graded categorical group
$(\mathbb{G},gr)$ determines a symmetric factor set   $\mathcal F $
on $\Gamma$ with coefficients in a  braided categorical group
$\Ker\mathbb G$.
\end{lemma}
\begin{proof}
 For each  $\sigma\in\Gamma$,   we construct a symmetric monoidal functor  $F^{\sigma} = (F^{\sigma},
    \widetilde{F}^{\sigma}):
\Ker \mathbb G\rightarrow \Ker \mathbb G$ as follows.
  For any $X\in \Ker \mathbb G$, since the grading $gr$ is stable, there is an isomorphism  $\Upsilon ^{\sigma}_{X} : X \stackrel\sim\rightarrow
  F^{\sigma}X$, where $F^{\sigma}X\in \Ker \mathbb G$, and
   $gr(\Upsilon ^{\sigma}_{X})=\sigma$. In particular, if $\sigma = 1$
    we set $F^{1}X = X$ and $\Upsilon^{1}_{X} = id_{X}$.
 For any morphism $f:X\rightarrow Y$ of grade 1 in $\Ker \mathbb G$, a morphism
   $F^{\sigma}(f)$ in $\Ker \mathbb G$ is  determined by
  %the following commutative diagram
%\begin{equation*}\label{mt}\begin{CD}
%X @> \Upsilon^{\sigma}_{X} >> F^{\sigma}X\\
%@V f VV           @VV F^{\sigma}(f) V\\
%Y @> \Upsilon^{\sigma}_{Y} >> F^{\sigma}Y.\\
%\end{CD}\end{equation*}
 $$F^{\sigma}(f)=\Upsilon^{\sigma}_{Y}\circ f\circ(\Upsilon^{\sigma}_{X})^{-1}.$$
 Natural isomorphisms $\widetilde{F}^{\sigma}_{X,Y} : {F^{\sigma}X}
\otimes {F^{\sigma}Y} \stackrel\sim\longrightarrow
F^{\sigma}(X\otimes Y)$ are  determined by %the commutativity of the
%following diagram
%\begin{equation*}\label{pn}
%\begin{diagram}
%\node[1]{F^{\sigma}X\otimes F^{\sigma}Y}
% \arrow[2]{e,t}{\widetilde{F}^{\sigma}_{X,Y}}
%\node[2]{F^{\sigma}(X\otimes Y).}
%\\
%\node[2]{X\otimes Y}\arrow[1]{nw,r,2}{{\Upsilon^{\sigma}_{X}}
%\otimes {\Upsilon^{\sigma}_{Y}}}
%\arrow[1]{ne,r,2}{\Upsilon^{\sigma}_{X\otimes Y}}
%\end{diagram}
%\end{equation*}
$$\widetilde{F}^{\sigma}_{X,Y}=(\Upsilon^{\sigma}_{X}
\otimes \Upsilon^{\sigma}_{Y})\circ(\Upsilon^{\sigma}_{X\otimes
Y})^{-1}.$$
Moreover, for any pair $\sigma,\tau\in\Gamma$, there is
an isomorphism between monoidal functors $ \theta^{\sigma,\tau}
:F^{\sigma}F^{\tau} \stackrel\sim\longrightarrow F^{\sigma\tau}$,
where $\theta^{1,\sigma} =id_{F^{\sigma}} =\theta^{\sigma,1}$, which
is determined by
%the commutativity of the following diagram,
$$\theta^{\sigma,\tau}_{X}=\Upsilon^{\sigma}_{F^{\tau}X}\circ\Upsilon^{\tau}_{X}\circ(\Upsilon^{\sigma\tau}_{X})^{-1},$$
for all $X \in \Ob\mathbb G$.

%\begin{equation*}\label{tt}\begin{CD}
%X @> \Upsilon^{\tau}_{X} >> F^{\tau}X\\
%@V \Upsilon^{\sigma\tau}_{X} VV           @VV \Upsilon^{\sigma}_{F^{\tau}X}V\\
%F^{\sigma\tau}X @< \theta^{\sigma,\tau}_{X} << F^{\sigma}F^{\tau}X.\\
%\end{CD}\end{equation*}
The pair $(F,\theta)$ determined above is a symmetric factor set.
\end{proof}
%-------------------------
\section {Braided crossed modules}
We first recall that a {\it crossed module} \cite{White49}   $(B,D,d,\vartheta)$
consists of groups $B,D$, group homomorphisms $d:B\ri D,\;\vartheta:D\ri$
Aut$B$ satisfying\\
\indent $C_1.\ \vartheta d=\mu$,\\
 \indent $C_2.\ d(\vartheta_x(b))=\mu_x(d(b)),\;x\in D,
b\in B,$\\
 where $\mu_x$ is an inner automorphism given
by conjugation with $x$.

In this paper,  the crossed module  $(B,D,d,\vartheta)$ is sometimes
denoted by $B\stackrel{d}{\ri}D$, or by $d:B\ri D$. For convenience, we write the addition for the operation in $B$ and  the multiplication for that in $D$.

The notion of braided crossed module over a groupoid was originally
introduced by Brown and Gilbert in \cite{Gilbert}. Later, the notion
of braided crossed module over groups appeared in the work of Joyal
and Street \cite{J-S} (Remark 3.1).%, of Calvo et al \cite{CCQ}
%(Example 6.13).

 \begin{definition} [\cite{J-S}] \emph{A} {\it braided crossed module} \emph{$\mathcal M$ is a crossed module
 $(B,D,d,\vartheta)$ together with a map $\eta: D\times D\ri B$ satisfying the following
 conditions:}

  $C_3.\ \eta(x,yz)=\eta(x,y)+\vartheta_y \eta(x,z),$

  $C_4.\ \eta(xy,z)=\vartheta_x\eta(y,z)+\eta(x,z),$

  $C_5.\ d\eta(x,y)=xyx^{-1}y^{-1},$

 $C_6.\ \eta(d(b),x)+\vartheta_xb=b,$

  $C_7.\ \eta(x,d(b))+b=\vartheta_xb,$\\
\emph{where   $b\in B$, $x,y,z\in D$.}
 \end{definition}

A braided crossed module  is called a  {\it symmetric} crossed module  (see Aldrovandi and Noohi \cite{ANoohi})
  if  $\eta(x,y)+\eta(y,x)=0$ for all $x,y\in
 D$. In this case, the conditions
  $C_3$ and $C_4$ coincide, the conditions $C_6$ and $C_7$ coincide.

The following properties follow from the definition of a braided crossed
module.
\begin{proposition}
Let $\mathcal M$ be a braided crossed module.

$\mathrm{i)}$ $\eta(x,1) = \eta(1,y) = 0$.

$\mathrm{ii)}$ $\Ker d$ is a subgroup of $Z(B)$.

$\mathrm{iii)}$ $\Coker d$ is an abelian group.

$\mathrm{iv)}$ The homomorphism $\vartheta$ induces the identity on $\Ker d$, and hence the action of $\Coker d$ on $\Ker d$, given by

\[sa=\vartheta_x(a),\ \ a\in \Ker d,  \ x\in s\in \Coker d,\]
is trivial.
\end{proposition}
The abelian groups $\Ker d$ and $\Coker d$  are also denoted by $\pi_1 \mathcal M$ and $\pi_0 \mathcal M$, respectively.
\begin{example} \label{vd} Let  $N$ be a normal subgroup of a group  $G$ so that the quotient group  $G/N$ is abelian, in other words, let $N$ be a normal subgroup in $G$ which contains the derived group (or the commutator subgroup) of $G$. Then, $(N,G,i,\mu,[,])$ is a braided crossed module, where $i:N\ri G$ is an inclusion, $\mu:G\ri \Aut N$ is defined by conjugation and  $\eta: G\times
G\ri N,$ $\eta(x,y)= [x,y] (= xyx^{-1}y^{-1})$.
 %In fact, it is well-known that the system $(N,G,i,\mu)$ is a crossed module. Since $N$ contains the commutator subgroup  of $G$,  $[x,y]\in N$, for all $x,y\in
%G$. By the definition of $[x,y]$, the condition $C_5$ holds. The conditions $C_3$, $C_4$, $C_6$, $C_7$ follow from the properties
%\begin{equation*}  [x,yz]=[x,y]\mu_y[x,z], \;\;[xy,z]=\mu_x[y,z][x,z],  \end{equation*}
%\begin{equation*} [x,y]y=\mu_x(y),\;\;[y,x]\mu_x(y)=y \end{equation*}
%of commutators, respectively.
\end{example}

According to Joyal and Street \cite{J-S}, each braided crossed module is determined by a braided strict categorical group. We now classify the category of  braided crossed modules.
\begin{definition} \emph{ A}  {\it homomorphism} \emph{of braided crossed modules $(B,D,d,\vartheta, \eta )$ and $(B',D',d',\vartheta', \eta')$ consists of group homomorphisms $f_1:B\ri B'$,
$f_0:D\ri D'$ such that:}

$H_1.\ f_0d=d'f_1$,

$H_2.\ f_1(\vartheta_xb)=\vartheta'_{f_0(x)}f_1(b)$,

$H_3.\ f_1(\eta(x,y))=\eta'(f_0(x),f_0(y)),$\\
\emph{for all $x,y\in D,\, b\in B$.}
\end{definition}
 Therefore, a homomorphism of braided crossed modules is that of crossed modules which satisfies  $H_3$.

 We  determine the category \[{\bf BrCross}\]
whose objects are braided crossed modules and whose morphisms are triples
$(f_1, f_0, \varphi)$, where $(f_1, f_0): (B\stackrel{d}{\rightarrow} D)\rightarrow
  (B'\stackrel{d'}{\rightarrow} D')$ is a homomorphism of braided crossed modules
   and $\varphi\in Z^2_{ab}(\Coker d, \Ker d')$. The composition with the morphism
    $(f'_1, f'_0, \varphi'):(B'\stackrel{d'}{\rightarrow} D')\rightarrow
  (B''\stackrel{d''}{\rightarrow} D'')$ is given by
\begin{equation}\label{ht}
(f_1',f_0',\varphi')\circ(f_1,f_0,\varphi)=(f_1'f_1,f_0'f_0,
(f'_1)_\ast\varphi+(f_0)^\ast\varphi').
\end{equation}

\begin{definition}  \emph{A symmetric monoidal functor $(F,\widetilde{F}):\mathbb G\ri\mathbb G'$ is termed}
  {\it regular  } \emph{if}

 $B_1.\ F(x)\otimes F(y)=F(x\otimes y),$

$B_2.\ F(b)\otimes F(c)=F(b\otimes c),$

 $B_3.\ \widetilde{F}_{x,y}=\widetilde{F}_{y,x},$\\
\emph{for $x,y\in$ Ob$\mathbb G,\ b,c\in$ Mor$\mathbb G.$}
\end{definition}

Denote by
 \[{\bf BrGr^*}\]
 the category of braided strict categorical groups
and regular  symmetric monoidal functors
 and denote by $p:D\ri \Coker d$ a canonical projection, we obtain the following classification result.
  \begin{theorem}[Classification Theorem]\label{pl} There exists an equivalence
\[\begin{matrix}
 \Phi:{\bf BrCross}&\ri&{\bf BrGr^*},\\
B\ri D&\mapsto&\mathbb G_{B\ri D}\\
(f_1,f_0,\varphi)&\mapsto&(F,\widetilde{F})
\end{matrix}\]
 where $F(x)=f_0(x),F(b)=f_1(b), \widetilde{F}_{x,y}=\varphi(px,py)$, for $x, y\in D, b\in B$.
\end{theorem}
\begin{proof}
 The proof of this theorem is a particular case of Theorem \ref{dlpl} in the next section.
\end{proof}

\begin{remark} Denote by $Br{\bf Cross}$ the subcategory of {\bf
BrCross} whose morphisms are homomorphisms of braided crossed
modules ($\varphi=0$) and denote by $Br{\bf Gr^*}$ the subcategory
of $ {\bf BrGr^*}$ whose morphisms are strict monoidal functors
($\widetilde{F}=id$). Then, these two categories are equivalent via
 $\Phi$.
\end{remark}
%-------------------------------------------

\section{ Braided $\Gamma$-crossed modules}
The main objective of this section is to classify braided
$\Gamma$-crossed modules
 by means of strict braided graded categorical groups.
First, observe that if $B$ is a $\Gamma $-group, the group $\Aut B$ of all automorphisms of $B$ is also a $\Gamma $-group under the action
\[(\sigma f)(b)=\sigma(f(\sigma^{-1}b)),\;b\in B,\;f\in \Aut B,\, \sigma\in \Gamma.\]
Then, the map $\mu:B\ri \Aut B, b\mapsto \mu_b$ ($\mu_b$ is an inner automorphism of $B$ given by conjugation with  $b$) is a homomorphism of $\Gamma $-groups.

 \begin{definition} \emph{Let $B$ and $D$ be $\Gamma $-groups. A } {\it braided (symmetric) $\Gamma$-crossed module},
 \emph{ is a braided (symmetric) crossed module   $\mathcal M = (B,D,d, \vartheta,\eta)$
   in which $d:B\ri D,\;\vartheta:D\ri$ Aut$B$ are  $\Gamma$-group homomorphisms satisfying the following
   conditions:}

$\Gamma _1.$  $\sigma(\vartheta_x(b))=\vartheta_{\sigma x}(\sigma
b),$

$\Gamma_2$.   $\sigma\eta(x, y)=\eta(\sigma x, \sigma y),$\\
\emph{where $\sigma\in\Gamma,\ x,y \in D $ and $b\in B.$}
\end{definition}

Braided (symmetric) $\Gamma$-crossed modules are also called braided
(symmetric) equivariant crossed modules by Noohi \cite{N}.

The following properties are implied from the definition of a braided $\Gamma$-crossed module.
\begin{proposition} \label{md2}
Let $\mathcal M $ be a braided $\Gamma $-crossed module.

 $\mathrm{i)}$ $\Ker d$ is a $\Gamma $-submodule of $Z(B)$.

 $\mathrm{ii)}$ $\Coker d$ is a $\Gamma $-module under the action
\[\sigma s=[\sigma x],  \, x\in s\in \Coker d,\, \sigma \in \Gamma.\]
\end{proposition}
\begin{example}  In Example \ref{vd}, if $G$ and $N$ are  $\Gamma$-groups, then $(N,G,i,\mu,[,])$ is a braided $\Gamma$-crossed module.
\end{example}
\begin{example}
Let  $d:B\rightarrow D$ be a morphism of $\Gamma$-module and let $D$
act trivially on  $B$. Let $\eta:D\times D\rightarrow Ker d$ be a
biadditive  function satisfies  $\Gamma_2$ and
$$\eta|_{Im d\times D}=0=\eta|_{D \times Im d}.$$
Then, $(B,D,d,0,\eta)$ is a braided $\Gamma$-crossed module.
\end{example}

 We now show
that braided $\Gamma$-crossed modules are determined by braided
strict  $\Gamma $-graded categorical groups. First, we say that
 a symmetric factor set $(F,\theta)$ on $\Gamma $ with coefficients in a   braided categorical group $\mathbb G$
 is \emph{regular} if $F^\sigma $ is a regular symmetric monoidal functor   and $\theta^{\sigma, \tau} =id$, for all $\sigma,\tau \in \Gamma $.
 \begin{definition} \emph{ A braided $\Gamma $-graded categorical group  $(\mathbb G,gr)$ is called } {\it strict}
 \emph{if}

\emph{i) $\Ker \mathbb G$ is a braided  strict categorical group,}

\emph{ii) $\mathbb G$ induces a regular symmetric factor set
$(F,\theta)$ on $\Gamma $ with coefficients in  $\Ker \mathbb G$.}
\end{definition}
Equivalently, a braided $\Gamma $-graded  categorical group
$(\mathbb G,gr)$ is \emph{strict} if it is a $\Gamma $-graded
extension of a braided strict  categorical group by a regular
symmetric factor set.

 $\bullet$ Constructing the braided  strict $\Gamma $-graded categorical group
$\mathbb G= \mathbb G_{\mathcal M}$ associated to a braided $\Gamma$-crossed module  $\mathcal M=(B,D,d, \vartheta,\eta)$.

 Objects of $\mathbb G$ are the elements of the group $D$, a $\sigma$-morphism
 $x\ri y$ is  a pair $(b,\sigma)$, where $b\in B,\sigma\in\Gamma$
such that $\sigma x=d(b)y$. The composition of two
morphisms is given by
\begin{equation}\label{lh1}(x\stackrel{(b,\sigma)}{\ri}y\stackrel{(c,\tau)}{\ri}z)=(x\xrightarrow{(\tau b+c, \tau\sigma)}z).
\end{equation}
Since $B$ is  a $\Gamma$-group, the composition is associative and unitary.

For each morphism $(b,\sigma)$ in $\mathbb G$, we have
\[(b,\sigma)^{-1}=(-\sigma^{-1}b,\sigma^{-1}),\]
and hence $\mathbb G$ is a groupoid.

The tensor operation on objects is given by the addition in
the group $D$ and, for two morphisms $(x\stackrel{(b,\sigma)}{\ri}y),
(x'\stackrel{(c,\sigma)}{\ri}y')$ in $\mathbb G$, we define
\begin{equation}\label{lh2}
(x\stackrel{(b,\sigma)}{\ri}y)\otimes(x'\xrightarrow{(c,\sigma)}y')=(xx'\xrightarrow{(b+\vartheta_yc,\sigma)}yy').
\end{equation}
The functoriality of the tensor operation is implied from the compatibility of the action
 $\vartheta$ with  the $\Gamma$-action and from the conditions in the definition of a
braided  $\Gamma$-crossed module.

%For morphisms $(x\stackrel{(b,\sigma)}{\ri}y
%\stackrel{(c,\tau)}{\ri}z), (x'\stackrel{(b',\sigma)}{\ri}y'
%\stackrel{(c',\tau)}{\ri}z')$ in $\mathbb G$, one has
%\[\begin{aligned}(x\stackrel{(b,\sigma)}{\ri}y
%\xrightarrow{(c,\tau)} z)\otimes (x'\xrightarrow{(b',\sigma)} y'
%\xrightarrow{(c',\tau)}
%z')&\stackrel{(\ref{lh1})}{=}(x\xrightarrow{(\tau b+c,\tau\sigma)}
%z)\otimes(x' \xrightarrow{(\tau b'+c',\tau\sigma)} z')\\
% \;&\stackrel{(\ref{lh2})}{=}(xx'\xrightarrow{(\tau b+c+\vartheta_z(\tau b'+c'),\tau\sigma)}
%zz'),\end{aligned}\]
%\[\begin{aligned} {[(x\xrightarrow{(b,\sigma)}y)\otimes(x'\xrightarrow{(b',\sigma)}y')]\circ} &
%[(y\xrightarrow{(c,\tau)} z)\otimes(y'\xrightarrow{(c',\tau)}z')]\\
% \stackrel{(\ref{lh2})}{=}&(xx'\xrightarrow{(b+\vartheta_yb',\sigma)}yy')\circ(yy'\xrightarrow{(c+\vartheta_zc',\tau)}zz')\\
%\stackrel{(\ref{lh1})}{=}&(xx'\xrightarrow{(\tau(b+\vartheta_yb')+c+\vartheta_zc',\tau\sigma)}zz').\end{aligned}\]
%Thus, the functoriality of tensor operation is equivalent to
%\[\tau
%b+c+\vartheta_z(\tau b'+c')=\tau(b+\vartheta_yb')+c+\vartheta_zc'.\]
%This holds since
%\[\begin{aligned}
%\tau(\vartheta_yb') \stackrel{(\Gamma_1)}{=} \vartheta_{\tau y}(\tau b')
%=\vartheta_{(dc) z}(\tau b')
%\stackrel{(C_1)}{=}\mu_c(\vartheta_{ z}(\tau
%b')).\end{aligned}\]
Associativity and unit constraints of the tensor operation are strict. The braiding constraint  $\bf c$ is defined by \[{\bf c}_{x,y}=(\eta(x,y),1):xy\ri yx.\]
By the relation $C_5$, ${\bf c}_{x,y}$ is actually a morphism in $\mathbb G$.
Due to the conditions $C_3, C_4$, the braiding constraint ${\bf c}$ is compatible  with the associativity constraint ${\bf a}$.  The naturality of ${\bf c}$ follows from the conditions $\Gamma _2, C_1,  C_3, C_4, C_6, C_7$.

The $\Gamma $-grading  $gr:  \mathbb G \rightarrow \Gamma$
is given by
\begin{align*} (b,\sigma)\mapsto \sigma. \end{align*}

The unit graded functor $I: \Gamma \rightarrow \mathbb G$
is defined by
\begin{align*} I(*\xrightarrow{\sigma } *)=(1\xrightarrow{(0,\sigma) } 1). \end{align*}

Since Ob$\mathbb G=D$ is a group and $x\otimes y=xy$, every object of $\mathbb G$ is invertible, and hence
 $\Ker\mathbb G$ is a braided strict categorical group.

We now show that $\mathbb G$ induces a regular symmetric factor set $(F,\theta)$ on $\Gamma$ with coefficients in $\Ker \mathbb G$.
For any $x\in D,\sigma\in\Gamma$, we set $F^\sigma(x)=\sigma x$,
  $\Upsilon ^\sigma_x =(x\stackrel {(0,\sigma)}{\rightarrow}\sigma x)$. Then, according to the proof of Lemma \ref{bd01}, we have $F^\sigma(b,1)=(\sigma b,1)$ and  $\theta^{\sigma,\tau}=id$. From the braided $\Gamma $-crossed module structure of $\mathcal M$, it follows that $F^\sigma $ is a regular symmetric monoidal functor on $\Ker \mathbb G$.

%Therefore, $\mathbb G$ is a braided strict $\Gamma $-graded  categorical group.

$\bullet$ Constructing the braided $\Gamma$-crossed module
 {\it associated} to a braided strict $\Gamma $-graded  categorical group $\mathbb G$.

Set
\begin{equation} D= \Ob\mathbb G,\;\;B=\{x\xrightarrow{b}1\ |\ x\in D,\ gr(b)=1  \}. \notag \end{equation}
 The operations in $D$ and $B$ are given by \[xy=x\otimes y,\;\ b+c=b\otimes c, \]
respectively. Then
 $D$ becomes a group in which the unity is 1 and the inverse of  $x$ is $x^{-1}$
($x\otimes x^{-1}=1$), $B$ is group in which the zero element is the
morphism $(1\xrightarrow{id_1} 1 )$ and the inverse of
$(x\xrightarrow{b} 1)$ is the morphism
$(x^{-1}\xrightarrow{\overline{b}}  1 )(b\otimes
\overline{b}=id_1)$. Since $\mathbb G$ has a regular symmetric
factor set $(F,\theta)$,   $D$ and $B$ are  $\Gamma $-groups under
the actions
 \begin{align*}
\sigma x&= F^\sigma (x),\; x\in D, \sigma \in \Gamma, \\
\sigma b&=F^\sigma (b),\ b\in B,
\end{align*}
respectively. The correspondences  $d:B\rightarrow D$ and $\vartheta: D\rightarrow
\Aut B$ are, respectively, given by
\[d(x\xrightarrow{b}  1)=x,\]
\[\vartheta _y(x\xrightarrow{b}1)= (yxy^{-1}\xrightarrow{id_y + b + id_{y^{-1}}}  1).\]
 Since $B$ and $D$ are  $\Gamma$-groups,   $d$ and $\vartheta$ are $\Gamma$-group homomorphisms.

The map $\eta:D\times D\ri B$ is defined by
\[\eta(x,y)={\bf c}_{x,y}\otimes id_{x^{-1}}\otimes id_{y^{-1}}: xyx^{-1}y^{-1}\rightarrow 1.\]

%It is easy to check that $(B,D,d,\vartheta,\eta)$ defined above is a braided $\Gamma $-crossed module.

Now we will classify the category of braided $\Gamma$-crossed
modules.
\begin{definition}  \emph{ A}  {\it homomorphism}
\emph{$\mathcal M \rightarrow  \mathcal M'$ of  braided
$\Gamma$-crossed modules is a homomorphism $(f_1,f_0)$  of braided
crossed modules, where $f_1,f_0$ are $\Gamma$-group homomorphisms.}

% $H_1.\ f_0d=d'f_1$,

 %$H_2.\ f_1(\vartheta_xb)=\vartheta'_{f_0(x)}f_1(b)$,

 %$H_3.\ f_1(\eta(x,y))=\eta'(f_0(x),f_0(y)),$\\
%for all $x,y\in D,\;\ b\in B$.
\end{definition}
\noindent \emph{Remark on notations}. Each morphism
  $x\xrightarrow{(b,\sigma)}y$ in
$\mathbb G_\mathcal M$ is written in the form
\[x\xrightarrow{(0,\sigma)}\sigma
 x\xrightarrow{(b,1)}y,\]
 and then each graded symmetric monoidal functor   $(F,\widetilde{F}):\mathbb G_{\mathcal M}\ri \mathbb G_{\mathcal M'}$ defines a function $f:D^2\cup (D\times \Gamma) \ri B'$ by
 \begin{equation}\label{01}
 (f(x,y),1)=\widetilde{F}_{x,y},\quad (f(x,\sigma),\sigma)=F(x\stackrel{(0,\sigma)}{\rightarrow}\sigma x).
 \end{equation}
\begin{lemma}\label{t1}
Let $(f_1,f_0): \mathcal M\rightarrow \mathcal M'$ be a homomorphism of braided $\Gamma $-crossed modules. Then there is a graded  symmetric monoidal functor $(F,\widetilde{F}): \mathbb G_{\mathcal M}\rightarrow \mathbb G_{\mathcal M'}$ such that $F(x)=f_0(x)$,
  $F(b,1)=(f_1(b),1)$, if and only if $f=p^{\ast}\varphi$,  where $\varphi \in Z^2_{\Gamma,s}(\Coker d,$ $ \Ker
d')$, and $p:D\rightarrow  \Coker d$ is a canonical projection.
 \end{lemma}
\begin{proof}
Since  $f_0$  is a homomorphism and $Fx=f_0(x)$, $\widetilde{F}_{x,y}: FxFy\ri
F(xy)$ is a morphism of grade 1 in $\mathbb G'$ if and only if
$df(x,y)=1'$, or $f(x,y)\in\Ker d'\subset Z(B')$.

Also, since $f_0$ is a $\Gamma$-homomorphism,  $Fx\xrightarrow{(f(x,\sigma),\sigma)}F\sigma x$ is a morphism of grade $\sigma$ in $\mathbb G'$ if and only if $df(x,\sigma)=1'$, or $f(x,\sigma)\in\Ker d'\subset
Z(B')$.

$\bullet$ The condition so that  $F$ preserves the composition of two morphisms.

 Since $f_1$ is a group homomorphism, $F$ preserves the composition of two morphisms of grade 1. $F$ preserves the composition of two morphisms in terms of $(0,\sigma)$,
 \[(x\xrightarrow{(0,\sigma) }y\xrightarrow{(0,\tau) }z),\]
if and only if
\begin{equation}\label{1s}\tau f(x,\sigma)+f(\sigma x,\tau)=f(x,\tau\sigma).
\end{equation}

$\bullet$ The condition so that $\widetilde{F}_{x,y}$ is natural.

- For morphisms of grade 1, we consider the diagram
\begin{equation*}\label{bdtn1}
\begin{diagram}
\node{F(x)F(y)}\arrow{e,t}{\widetilde{F}_{x,y}}\arrow{s,l}{F(b,1)\otimes
F(c,1)}\node{F(xy)}\arrow{s,r}{F[(b,1)\otimes (c,1)]}\\
\node{F(x')F(y')}\arrow{e,b}{\widetilde{F}_{x',y'}}\node{F(x'y')}
\end{diagram}
\end{equation*}
Since  $f_1,\ f_0$ are homomorphisms satisfying the condition $H_2$,
\[F(b,1)\otimes
F(c,1)= F[(b,1)\otimes
(c,1)].\]
Then since
 $f(x,y),\ f(x',y')\in Z(B')$, the above diagram commutes if and only if
\[f(x,y)=f(x',y'),\]
for $x=d(b)x'$, $y=d(c)y'.$ Thus, $\widetilde{F}$ defines a function
$\varphi:\Coker^2 d\ri\Ker d'$,
\[\varphi(r,s)=f(x,y),\  r=px, s=py.\]

- For morphisms in terms of $(0,\sigma)$, we consider a diagram
\begin{equation*}\label{bdtn2}
\begin{diagram}
\node{F(x)F(y)}\arrow{e,t}{\widetilde{F}_{x,y}}\arrow{s,l}{F(0,\sigma)\otimes
F(0,\sigma)}\node{F(xy)}\arrow{s,r}{F [ (0,\sigma)\otimes (0,\sigma )]}\\
\node{F(\sigma x)F(\sigma y)}\arrow{e,b}{\widetilde{F}_{\sigma
x,\sigma y}}\node{F(\sigma x)(\sigma y)=F\sigma(xy).}
\end{diagram}
\end{equation*}
According to Proposition \ref{md2}, the above diagram commutes if and only if
\begin{equation}\label{2s} \sigma f(x,y) +f(xy,\sigma)=f(x,\sigma)+ f(y,\tau) +f(\sigma x,\sigma
y).
\end{equation}
$\bullet$ Since the following square commutes
\begin{equation*}\label{bdtn3}
\begin{diagram}
\node{Fx}\arrow{e,t}{(f(x,\sigma),\sigma)}\arrow{s,l}{F(b,1) }\node{F\sigma x}\arrow{s,r}{F(\sigma b,1)}\\
\node{Fy}\arrow{e,b}{(f(y,\sigma),\sigma)}\node{F\sigma y}
\end{diagram}
\end{equation*}
and   $f_1$ is a $\Gamma$-group homomorphism, we have
$f(x,\sigma)=f(y,\sigma),$ for $x=d(b)y$. This determines a function
$\varphi:\Coker d\times\Gamma\ri \Ker d'$,
\[\varphi(r,\sigma)=f(x,\sigma),\ r=px.\]
Therefore, we obtain a function
\[\varphi:\Coker^2d\cup\Coker d\times\Gamma\ri\Ker d'.\]
The function $\varphi$ is normalized in the sense that
\[\varphi(1,r)=\varphi(s,1)=0=\varphi(s,1_\Gamma).\]
The first two equalities follow from the property $F(1)=1'$ and the compatibility of $(F,\widetilde{F})$
with unit constraints. The final equality holds owing to $f(x,1_\Gamma )=0$ (following from the relation \eqref{1s}).

By Proposition \ref{md2}, the compatibility of $(F,\widetilde{F})$
with associativity constraints is equivalent to
\begin{equation}\label{3s} f(y,z) +f(x,yz)=f(x,y)+f(xy,z).\end{equation}

The compatibility of $(F,\widetilde{F})$ with the braiding constraints implies
\begin{align*}   f(x,y)+f_1(\eta(x,y))= \eta'(f_0(x),f_0(y))+ f(y,x).
 \end{align*}
By the fact that $f(x,y) \in \Ker d'\subset Z(B')$ and by  the condition $H_3$, one has
\begin{equation} \label{4s} f(x,y)=f(y,x).\end{equation}
From the relations  (\ref{1s})--(\ref{4s}), it follows that
 $\varphi\in Z^2_{\Gamma, s}(\Coker d,\Ker d')$.
\end{proof}
 We define the category
 \[{\bf _\Gamma BrCross}\]
 whose objects are braided $\Gamma$-crossed modules and whose morphisms are triples \linebreak $(f_1, f_0, \varphi)$, where $(f_1, f_0): (B\stackrel{d}{\rightarrow} D)\rightarrow
  (B'\stackrel{d'}{\rightarrow} D')$ is a homomorphism of  braided $\Gamma$-crossed modules,
   and $\varphi \in Z^2_{\Gamma,s}(\Coker d, \Ker d')$. The
   composition is given by \eqref{ht}.

Note that a braided strict $\Gamma $-graded  categorical group
$\mathbb G$ induces $\Gamma$-actions on the group $D$ of objects and
on  the group $B$ of morphisms of grade 1, we state the following
definition.
  \begin{definition} \emph{ A graded symmetric monoidal functor  $(F,\widetilde{F}):\mathbb G\ri\mathbb G'$ is termed}
   {\it regular } \emph{if}

 $B_1.\ F(x)\otimes F(y) = F(x\otimes y),$

 $B_2.\ F( b)\otimes F( c)=F( b\otimes c),$

 $B_3.\ \widetilde{F}_{x,y}=\widetilde{F}_{y,x}$,

 $B_4.\ F(\sigma x)=\sigma F(x)$,

 $B_5.\ F(\sigma b)=\sigma F(b)$,\\
\emph{ where $x,y\in$ Ob$\mathbb G,$ $b,c$ are morphisms of grade 1
in $\mathbb G, \sigma\in \Gamma$.}
\end{definition}
The  graded symmetric monoidal functor mentioned in Lemma \ref{t1} is regular.

  \begin{lemma}\label{n1}
Let  $\mathbb G$, $\mathbb G'$ be corresponding braided strict $\Gamma $-graded  categorical groups associated to  braided $\Gamma$-crossed modules
$\mathcal M$, $\mathcal M'$, and let $(F,\widetilde{F}):\mathbb G \ri \mathbb G'$ be a regular graded symmetric monoidal functor.
Then, the triple  $(f_1,f_0,\varphi)$, where\\
\indent \emph{i)} $f_0(x)=F(x),\ (f_1(b),1)=F(b,1),\  \sigma\in\Gamma,b\in B, x \in D,$\\
\indent \emph{ii)} $p^*\varphi=f$, where $f$ is defined by  \eqref{01},\\
is a morphism in the category ${\bf _\Gamma BrCross}$.
\end{lemma}

%%%%%%%%%%%%%%%%%%%%%%%%%5
\begin{proof}
Due to the conditions  $B_1$ and $B_4$, $f_0$ is a $\Gamma $-group homomorphism.
By the assumption that $F$ preserves the composition of two morphisms of grade 1 and by the condition $B_5$, $f_1$ is  a $\Gamma$-group homomorphism.
Any $b\in B$ can be considered as a morphism
$(db\stackrel{(b,1)}{\ri} 1)$ in $\mathbb G$, and hence
$(f_0(db)\stackrel{(f_1(b),1)}{\ri}1')$ is a morphism in
$\mathbb G'$, that is, the relation $H_1$ holds.
The relation $H_2$ follows from the condition $B_2$ and the homomorphism property of $f_1$.

According to the proof of Lemma \ref{t1}, the compatibility of $(F,\widetilde{F})$ with braiding constraints and the condition $B_3$ lead to the relation $H_3$.
So, $(f_1,f_0)$ is a homomorphism of braided crossed  $\Gamma$-modules. Thus, by Lemma \ref{t1},
the function $f$ determines a function
 $\varphi\in Z^2_{\Gamma,s}(\Coker d, \Ker d')$ such that $f=p^{*}\varphi,$ where $ p:D\ri \Coker d$ is a canonical projection.
Therefore,  $(f_1,f_0,\varphi )$ is a morphism in ${\bf _\Gamma BrCross}$.
\end{proof}

%%%%%%%%%%%%%%%%%%%%%%5

Denote by
 \[{\bf _\Gamma BrGr^*}\]
 the category of braided strict $\Gamma $-graded  categorical groups and regular graded symmetric monoidal functors,
we obtain the following result which is an extension of  Theorem \ref{pl}.
\begin{theorem}[\label{dlpl}Classification Theorem] There exists an equivalence
\[\begin{matrix}
 \Phi:{\bf _\Gamma BrCross}&\ri&{\bf _\Gamma BrGr^*},\\
B\ri D&\mapsto&\mathbb{G}_{B\ri D}\\
(f_1,f_0, \varphi )&\mapsto&(F,\widetilde{F})
\end{matrix}\]
 where $\ F(x)=f_0(x)$, $F(b,1)=(f_1(b),1)$, $F(x\stackrel{(0,\sigma)}{\rightarrow}\sigma x)=(\varphi (px,\sigma),\sigma )$,  $\widetilde{F}_{x,y}=$\linebreak$(\varphi (px,py), 1)$,
for $x\in D, b\in B,\sigma\in\Gamma$.
\end{theorem}
\begin{proof}
Suppose that $\mathbb G,\mathbb G'$ are braided strict $\Gamma $-graded  categorical groups  associated to braided $\Gamma $-crossed modules  ${B\ri D}, {B'\ri D'}$, respectively.
By Lemma \ref{t1}, the correspondence $(f_1, f_0, \varphi)\mapsto (F,\widetilde{F})$ determines an injection on
the homsets
\[\Phi:\Hom_{{\bf _\Gamma BrCross}}(B\ri D, B'\ri D')\ri \Hom_{{\bf _\Gamma BrGr^*}}(\mathbb G,\mathbb G').\]
According to Lemma \ref{n1}, $\Phi$ is  surjective.

 If $\mathbb G$ is a braided strict $\Gamma $-graded  categorical group, and $\mathcal M_{\mathbb G}$ is its an associated braided $\Gamma $-crossed module, then $\Phi(\mathcal M_{\mathbb G})=\mathbb G$ (not only isomorphic). Therefore, $\Phi$ is an equivalence.
\end{proof}

 \begin{remark} In the above theorem, if $B\rightarrow D$ is a symmetric $\Gamma $-crossed module, then $\mathbb G_{B\rightarrow D}$ is a symmetric strict $\Gamma $-graded  categorical group. Let $\bf _\Gamma SymCross$ denote the full subcategory of the category  $\bf _\Gamma BrCross$ whose objects are symmetric crossed  $\Gamma$-modules, and let $\bf _\Gamma PiGr^*$ denote the full  subcategory of the category $\bf _\Gamma BrGr^*$  whose objects are symmetric strict $\Gamma $-graded categorical groups. Then these two subcategories are equivalent and the following  diagram commutes
 \begin{align*}
\begin{diagram}\xymatrix{{\bf _\Gamma SymCross} \ar[r]^{\Phi}\ar@{^(->}[d]_{J}& {\bf _\Gamma PiGr^*}\ar@{^(->}[d]^{J^*}\\
{\bf _\Gamma BrCross}\ar[r]_{\Phi}& {\bf _\Gamma BrGr^*},
 }\end{diagram}\end{align*}
where $J, J^*$ are full embedding functors.
 \end{remark}
 \begin{remark}
   When $\Gamma=1$ is a trivial group, then the categories  ${\bf _\Gamma BrCross}$ and $ {\bf _\Gamma BrGr^*}$ are the categories  ${\bf  BrCross}$ and ${\bf  BrGr^*}$, respectively. Therefore, we obtain Theorem \ref{pl}.
 \end{remark}
 %--------------------------------

 \section {Classification of $\Gamma $-module extensions of the type of an abelian $\Gamma $-crossed module}
In this section, we present the theory of $\Gamma$-module extension of the type of an abelian $\Gamma$-crossed modules, which is analogous to the theory of group extension of the type of a crossed module \cite{Tay,Ded,Br94}.

In \cite{CCG}, if  $d:B\ri D$ is a homomorphism of abelian groups and  $D$ acts trivially on $B$, then
 $(B,D,d,0)$ is called an {\it abelian crossed module}. Let us note that any abelian crossed module is defined by a strict Picard category, that is, a symmetric categorical group in which ${\bf a}=id,$ ${\bf c}=id,$ ${\bf l}=id={\bf r}$ and for each object $x$, there is an object $y$ such that $x\otimes y=1).$

By an \emph{abelian} $\Gamma$-crossed module, we shall mean a braided  $\Gamma$-crossed module  \linebreak  $(B,D,d,\vartheta,\eta)$ that  $\vartheta=0$, $\eta=0$. Then $d$ is a homomorphism of $\Gamma $-modules.

 According to the construction in Section 4,
  each abelian $\Gamma$-crossed module $\mathcal M = (B,D,d)$ defines a $\Gamma $-graded category $\mathbb G_{\mathcal M}$ whose $\Ker \mathbb G$ is a strict Picard category. In this case, we say that $\mathbb G_{\mathcal M}$ is a  strict $\Gamma $-graded Picard category.
A \emph{homomorphism} $(f_1,f_0): (B,D,d)\rightarrow (B',D',d')$ of abelian $\Gamma$-crossed modules consists of   $\Gamma $-module homomorphisms  $f_1: B\rightarrow B'$ and $f_0: D\rightarrow D'$ such that
\[f_0d=d'f_1.\]

Note that in this section, since $B$ and $D$ are abelian groups, we write $+$ for the operations on $B$, $D$.
 \begin{definition}
\emph{Let $\mathcal M=(B, D,d)$ be an abelian $\Gamma$-crossed
module, and let  $Q$ be a $\Gamma$-module. A}  \emph{$\Gamma$-module
extension}  \emph{of $B$ by $Q$}  {\it of type} $\mathcal M$,
\emph{denoted by $\mathcal E_{d,Q}$, is a short exact sequence of
$\Gamma $-module homomorphisms,}
\[\mathcal E: 0\rightarrow B\stackrel{j}{\rightarrow}E\stackrel{p}{\rightarrow}Q\rightarrow 0,\]
\emph{and a homomorphism of abelian $\Gamma $-crossed modules
$(id,\varepsilon): (B\rightarrow E)\rightarrow (B\rightarrow D)$.}
\end{definition}

Two   extensions $\mathcal E_{d,Q}$ and $\mathcal E'_{d,Q}$
are said to be \emph{equivalent} if the following diagram commutes
\begin{align*}  \begin{diagram}
\xymatrix{ \mathcal E: 0 \ar[r]& B \ar[r]^j \ar@{=}[d] &E  \ar[r]^p
\ar[d]^ \alpha  & Q \ar[r]\ar@{=}[d] & 0,&\;\;\;\;E
\ar[r]^\varepsilon&D \\
 \mathcal E: 0 \ar[r]& B \ar[r]^{j'}  &E'  \ar[r]^{p'}   & Q \ar[r]& 0,&\;\;\;\;E
\ar[r]^{\varepsilon'}&D}
\end{diagram}
\end{align*}
and $\varepsilon' \alpha =\varepsilon$. Obviously, $ \alpha  $ is
an  isomorphism of $\Gamma $-modules.

Each extension $\mathcal E_{d,Q}$
induces a $\Gamma $-module homomorphism  $\psi: Q\rightarrow \Coker d
$  such
that $\psi p= q\varepsilon$, where $q: D\rightarrow \Coker d$ is a canonical projection.
Moreover, $\psi$  is
dependent only on the equivalence class of the extension $\mathcal E_{d,Q}$, and then we say that $\mathcal E_{d,Q}$ \emph{induces} $\psi$. The set
 of equivalence
classes of extensions $\mathcal E_{d,Q}$
 inducing $\psi:Q\rightarrow \Coker d$ is denoted by
\[\Ext_{\mathbb Z\Gamma }^{\mathcal M}(Q,B,\psi).\]

Now, in order to study this set we apply the obstruction theory for graded symmetric monoidal functors between  strict $\Gamma $-graded  Picard categories  $ \Dis_{\Gamma,s} Q$ and $\mathbb G_{B\ri D}$, where the  {\it discrete} $\Gamma $-graded  Picard category   $ \Dis_{\Gamma,s} Q$ is defined by
 (see Subsection 2.2)
\[\Dis_{\Gamma,s} Q =\int_{\Gamma}(Q, 0, 0).\]
It is just the strict $\Gamma $-graded  Picard category associated to the abelian $\Gamma $-crossed module $(0,Q,0)$ (see Section 4).
\begin{lemma}\label{mrtc}
Let $\mathcal M = (B, D, d)$ be an abelian $\Gamma $-crossed module, $Q$ be a $\Gamma $-module and
$\psi: Q\rightarrow  \Coker d$ be a $\Gamma $-module homomorphism. Then for each graded symmetric monoidal functor
$(F,\widetilde{F}):\;\Dis _{\Gamma,s}Q\rightarrow  \mathbb G_{\mathcal M}$ which satisfies $F(0)=0$ and induces the pair of  $\Gamma $-module homomorphisms $(\psi,0):(Q,0)\rightarrow  (\Coker d,\Ker
d)$, there exists an extension $\mathcal E_{d,Q}$
inducing $\psi$.
\end{lemma}
Such an extension $\mathcal E_{d,Q}$ is called
\emph{associated} to the graded symmetric monoidal functor $(F,\widetilde{F})$.
\begin{proof}

Suppose that $(F,\widetilde{F}):\; \Dis_{\Gamma,s} Q\rightarrow  \mathbb G_\mathcal M$ is  a  graded  symmetric  monoidal functor. Then $(F,\widetilde{F})$ determines a function  $f: Q^2 \cup ( Q\times \Gamma )\rightarrow
B$ by \eqref{01},
\[(f(u,v),1 )= \widetilde{F}_{u,v},\;\;(f(u,\sigma),\sigma )=F(u\stackrel{(0,\sigma)}{\longrightarrow} \sigma u ).\]

The function $f$  is ``normalized'' in the sense that
\begin{align*}\label{eq54} f(u,1_\Gamma)=0,f(u,0)=0=f(0,v).\end{align*}

Since $F$ preserves the identity morphism, one has the first equality. The later equalities follow from the assumption $F(0)=0$ and the compatibility of $(F,\widetilde{F})$ with unit constraints.
It follows from the definition of a morphism in $\mathbb G$ that
\begin{equation} \sigma F(u)=df(u,\sigma)+F(\sigma u),\label{mt1}
\end{equation}
\begin{equation} F(u)+F(v)=df(u,v)+F(u+v).\label{mt2}
\end{equation}
The function $f$ defined as above is just a 2-cocycle in $Z^2_{\Gamma ,s}(Q,B)$. %Indeed, the
%compatibilities of $(F,\widetilde{F})$ with the associativity   and
%commutativity constraints of two categories are equivalent to the
%relations \eqref{eq4.8}, \eqref{eq4.9} in which the left sides are
%equal to $0$ and $g$ is replaced by $f$. The naturality of
%$\widetilde{F}_{u,v}$ is equivalent to the relation \eqref{eq4.10}.
%Since $F$ preserves the composition of morphisms, we have the
%relation \eqref{eq4.11}.

From the 2-cocycle $f$, we construct an exact
sequence of $\Gamma$-modules
 \[\mathcal E_F:\;\ 0\rightarrow  B\stackrel{j_0}{\rightarrow}E_0\stackrel{p_0}{\rightarrow}Q\rightarrow  0,\]
 where $E_0$ is the crossed product  extension  $B\times_f Q$ and
$ j_0(b)=(b,1)$,  $p_0(b,u)= u,$ for $b\in B, u\in Q.$ The $\Gamma $-module structure of $E_0$ is given by
\[(b,u)+(c,v)= (b+c + f(u,v),u+v),\]
\[\sigma (b,u)=(\sigma b+f(u,\sigma),\sigma u).\]

Now we determine $\Gamma$-module homomorphism $\varepsilon:E_0\ri D$. By the assumption, $(F,\widetilde{F})$ induces a $\Gamma$-module homomorphism
$\psi:Q\ri \Coker d $  by $\psi(u)=[Fu]\in \Coker d$. Thus, the element $Fu$ is a representative of $\Coker d$ in $D$.
Then for $(b,u)\in E_0$, we set
\begin{equation}\label{ep}
\varepsilon(b,u)=db+Fu.
\end{equation}
Therefore, $\varepsilon$ is a
 $\Gamma $-module homomorphism thanks to the relations (\ref{mt1}) and (\ref{mt2}).
It is easy to see that $\varepsilon\circ j_0=d$. Further, this extension induces $\Gamma $-module homomorphism $\psi: Q\rightarrow \Coker d$, since
\[\varepsilon(b,u)=q(db+Fu)= q(Fu)= \psi(u)=\psi p_0(b,u),\]
for all $u\in Q$.
 \end{proof}

\begin{theorem}[Schreier  Theory  for $\Gamma $-module extensions of the type of an abelian $\Gamma $-crossed module] \label{schr} Let  $\mathcal M=(B,D,d)$ be an abelian $\Gamma $-crossed module, and let $\psi: Q\rightarrow \Coker d$ be  a $\Gamma$-module homomorphism. There exists a bijection
\[
\Omega:\mathrm{Hom}_{(\psi,0)}[\Dis_{\Gamma,s} Q,\mathbb G_{\mathcal M}]\rightarrow \mathrm{Ext}^{\mathcal M}_{\mathbb Z\Gamma}(Q, B,\psi).\]
\end{theorem}
\begin{proof}
{\it Step 1: Graded symmetric monoidal functors $(F,\widetilde{F}) $,
$(F',\widetilde{F}') $ are homotopic if and only if the
corresponding associated extensions $\mathcal E_{d,Q}, \mathcal E'_{d,Q}$  are equivalent.}

Suppose that $F, F':\;\Dis_{\Gamma,s} Q\rightarrow \mathbb G_\mathcal M$ are homotopic
by a homotopy  $\alpha:F\rightarrow  F'$. Then, there is a function  $g:Q\rightarrow  B$ such that $\alpha_u=(g(u),1)$, that is,
\begin{equation}\label{dkmt}
F(u)=dg(u)+F'(u).
\end{equation}
The naturality and the coherence condition (\ref{3.4}) of the homotopy $\alpha$  lead to $g(0)=0$ and
\begin{equation}\label{ttn}
f(u,\sigma)+g(\sigma u)=\sigma g(u)+f'(u,\sigma),
\end{equation}
\begin{equation}\label{dkk}
f(u,v)+g(u+v)=g(u)+ g(v)+f'(u,v).
\end{equation}

According to Lemma \ref{mrtc}, there exist the
extensions $\mathcal E_{d,Q}$ and $ \mathcal E'_{d,Q}$ associated to $F$ and $F'$, respectively.
Then, thanks to the relations (\ref{ttn}) and (\ref{dkk}), the map
\begin{align*} \alpha^\ast: E_F\rightarrow E_{F'},\quad
(b,u)\mapsto (b+g(u),u) \end{align*}
is a homomorphism of $\Gamma$-modules. Further, $\alpha^\ast$  is an isomorphism. The equality $\varepsilon'\alpha^\ast=\varepsilon$ is implied from the relations \eqref{ep} and \eqref{dkmt}:
\begin{align*}\varepsilon'\alpha^\ast(b,u)&=\varepsilon'(b+g(u),u)=d(b+g(u))+F'u\\
&=d(b)+d(g(u))+F'u=d(b)+Fu=\varepsilon(b,u).
\end{align*}
Therefore, two extensions $\mathcal E_{d,Q}$ and $ \mathcal E'_{d,Q}$ are equivalent.

Now, suppose that $\mathcal E_{d,Q}$ and $ \mathcal E'_{d,Q}$ are two extensions associated to $(F,\widetilde{F})$ and $(F',\widetilde{F}')$, respectively.
If $\alpha^\ast: E_F\rightarrow  E_{F'}$ is an equivalence  of these extensions, then it is straightforward to see that
 \[\alpha^\ast(b,u)=(b+g(u),u),\]
  where $g:Q\rightarrow
B$ is a function with $g(0)=0$.
By retracing our steps, $\alpha_u=(g(u),1)$ is a homotopy between $(F,\widetilde{F})$  and $(F',\widetilde{F}')$.

{\it Step 2: $\Omega$ is surjective.}

 Assume that $\mathcal E = \mathcal E_{d,Q}$ is an extension of type $\mathcal M$.
  We prove that  $\mathcal E$  defines
  a graded symmetric  monoidal functor $(F,\widetilde{F}):\;\Dis_{\Gamma,s} Q\ri \mathbb G_{\mathcal M}$.
For any $u\in Q$, choose a representative $e_u\in E$ such that
$p(e_u)=u,\;e_0=0$.  Each element of $E$ can be represented uniquely  as $b+e_u $ for $b\in B, u\in Q$.
The representatives $\left\{e_u\right\}$ induce a normalized function $f:Q^2\cup (Q\times\Gamma)\ri B$ by
\begin{equation}    e_u+e_v=f(u,v)+e_{u+v},\label{eq3'}\end{equation}
\begin{equation}
\sigma e_u=f(u,\sigma )+e_{\sigma u}.\label{eq4'} \end{equation}

 Now, we construct a graded symmetric monoidal functor  $(F,\widetilde{F}):$\linebreak
$\Dis_{\Gamma,s} Q\ri \mathbb G_{\mathcal M}$ as follows. Since $\psi(u)=\psi p(e_u)=q\varepsilon(e_u)$, $\varepsilon(e_u)$ is a representative of $\psi(u)$ in $D$. Thus, we set
 \[F(u)=\varepsilon(e_u),\;\;F(u\stackrel{\sigma}{\rightarrow}\sigma u)=(f(u,\sigma ),\sigma),\;\; \widetilde{F}_{u,v}=(f(u,v),1).\]
 The relations \eqref{eq3'} and \eqref{eq4'} show that
  $F(\sigma )$ and $\widetilde{F}_{u,v} $ are morphisms in $\mathbb G$, respectively.
 The associativity and commutativity  laws and the $\Gamma $-group property of  $B$ show that $f\in Z^2_{\Gamma ,s}(Q,B)$, and hence $(F,\widetilde{F})$  is a graded symmetric monoidal functor of type
 $(\psi,0)$.

 Now, let $\mathcal E_F$ be  an extension associated to $(F,\widetilde{F})$, then $\mathcal E_F \cong \mathcal E$ by $\alpha:(b,u)\mapsto b+ e_u$.
 \end{proof}
Let $\mathbb G$ be a $\Gamma $-graded Picard category associated to
an abelian $\Gamma$-crossed module $B\stackrel{d}{\rightarrow}D$.
Since $\pi_0\mathbb G=\Coker d$ and $\pi_1\mathbb G=\Ker d$, it
follows from  Subsection 2.2 that the reduced $\Gamma $-graded
Picard category $\mathbb  G(h)$ of $\mathbb G$ is of the form
 \[\mathbb G(h)=\int_{\Gamma}(\Coker d,\Ker d,h), \; h\in
Z^{3}_{\Gamma,s}(\Coker d,\Ker d).\]
Then $\Gamma $-module homomorphism $\psi: Q\rightarrow \Coker d$ induces an
 obstruction
\[\psi^\ast h\in Z^3_{\Gamma ,s}(Q,\Ker d).\]

We now use this  notion of obstruction to state and prove the following
theorem.
\begin{theorem}\label{dlc} Let  $\mathcal M =(B,D,d)$ be an abelian $\Gamma $-crossed module, and let
 $\psi: Q\rightarrow \Coker d$ be a  homomorphism of $\Gamma $-modules. Then, the vanishing of $\overline{\psi^\ast h}$ in $H^3_{\Gamma,s} (Q, \Ker d)$ is  necessary and sufficient for there to exist an extension $\mathcal E_{d,Q}$ of type  $\mathcal M$ inducing $\psi$. Further, if $\overline{\psi^\ast h}$ vanishes, then the set of equivalence classes of such extensions is bijective with
$H^2_{\Gamma,s}(Q,\Ker d)$.
\end{theorem}
\begin{proof}
By the assumption  $\overline{\psi^\ast
h}=0$, it follows from Proposition \ref{dl2.2a}
that
there is a graded symmetric monoidal functor $(\Psi,\widetilde{\Psi}):{\Dis_{\Gamma,s} Q}\rightarrow  \mathbb G(h)$. Then the composition of $(\Psi,\widetilde{\Psi})$ and the canonical  graded symmetric monoidal functor
 $(H,\widetilde{H}): \mathbb G(h)\rightarrow  \mathbb G$    is a graded symmetric monoidal functor
$(F,\widetilde{F}):\;\Dis_{\Gamma,s} Q\rightarrow
\mathbb G$, and hence by Lemma \ref{mrtc}, we obtain an associated extension
$\mathcal E_{d,Q}$.

Conversely, suppose that \[\mathcal E: 0\rightarrow B\stackrel{j}{\rightarrow} E \stackrel{p}{\rightarrow}Q\rightarrow 0\]
is a $\Gamma $-module extension of  type  $\mathcal M$ inducing $\psi$.
  Let $ \mathbb G'$  be a  strict $\Gamma $-graded Picard category associated to the
abelian $\Gamma $-crossed module
$(B, E,j)$. Then, according to Proposition \ref{t1}, there is a
  graded symmetric monoidal functor $F:\mathbb G' \rightarrow  \mathbb G$. Since
the reduced $\Gamma $-graded Picard category of $\mathbb G'$  is $\Dis_{\Gamma,s} Q$, it follows from Proposition \ref{dl2.1a} that
  $F$ induces a
graded symmetric monoidal functor of type
$(\psi,0)$ from $(Q,0,0)$ to $(\Coker d,\Ker d,h)$. Now,
thanks to Proposition \ref{dl2.2a}, the obstruction of the pair $(\psi,0)$ vanishes  in $H^3_{\Gamma,s}(Q, \Ker d),$ i.e., $\overline{\psi^\ast
h}=0$.

The final assertion  of Theorem \ref{dlc}  is obtained  from Theorem \ref{schr}.  First, there is a natural bijection
\[\Hom[\Dis_{\Gamma,s} Q, \mathbb G]\leftrightarrow \Hom\Dis_{\Gamma,s} Q, \mathbb G(h)].\]
Then, since $\pi_0(\Dis_{\Gamma,s} Q)=Q, \pi_1\mathbb G(h)=\Ker
d$, the bijection
\[\mathrm{Ext}_{\mathbb Z\Gamma}^{\mathcal M}(Q, B,\psi)\leftrightarrow
H^2_{\Gamma ,s}(Q,\Ker d)\]
follows from Theorem \ref{schr} and Proposition \ref{dl2.2a}.
\end{proof}

We now consider the special case when $\mathcal M=(B, \Aut B, 0)$ is an abelian $\Gamma$-crossed module. Then, each $\Gamma$-module extension of type $\mathcal M$ inducing $\psi:Q\ri \Aut B$ is just an extension of $\Gamma$-modules,
\[0\ri B\ri E\ri Q\ri 0,\]
inducing $\psi$.
Thus, Theorem \ref{dlc} leads to the following consequence.
\begin{corollary} [\cite{CK2007}, Theorem 2.4]  Let $B,Q$ be $\Gamma$-modules, and let $\psi:Q\ri \Aut B$ be a  $\Gamma$-module homomorphism.  Then, there is an obstruction  class $\overline{k} \in H^3_{\Gamma ,s}  (Q, B)$ whose vanishing is  necessary and sufficient for there to exist a $\Gamma $-module extension of $B$ by $Q$ inducing $\psi$. Further, if
$\overline{k} $ vanishes, then there exists a bijection
\[\Ext_{\mathbb Z\Gamma}(Q,B,\psi)\leftrightarrow H^2_{\Gamma,s}(Q,B).\]
\end{corollary}
%In particular, if $\Gamma=1$, the trivial group, then each $\Gamma$-module extension of type $\mathcal M$ inducing $\psi:Q\ri \Aut B$ is just an abelian group extension of $B$ by $Q$ inducing $\psi$. Then we obtain a bijection
%\[\Ext_{\mathbb Z}(Q,B,\psi)\leftrightarrow H^2_s(Q,B).\]

{}

\end{document}